\documentclass[12pt,a4paper]{amsart}

\usepackage{amsfonts, amsmath, amssymb, amsthm, amscd}
\usepackage{anysize}
\usepackage{graphicx}

\marginsize{2.1cm}{2.1cm}{2.2cm}{2.2cm}

\newtheorem{theorem}{Theorem}
\newtheorem{lemma}[theorem]{Lemma}
\newtheorem*{claim*}{Claim}

\newtheorem{proposition}[theorem]{Proposition}

\theoremstyle{definition}
\newtheorem{definition}[theorem]{Definition}

\theoremstyle{remark}
\newtheorem{remark}[theorem]{Remark}

\newcommand\Q{\mathbb Q}
\newcommand\Z{\mathbb Z}
\newcommand\R{\mathbb R}
\newcommand\G{\mathbb G}
\newcommand\C{\mathbb C}

\newcommand{\RP}{\mathbb P_{\mathbb R}^2}
\newcommand{\CP}{\mathbb P_{\mathbb C}^2}

\newcounter{fig}
\newcommand{\f}{\refstepcounter{fig} Fig. \arabic{fig}. }

\title{Two circles and only a straightedge}
\author{Arseniy~Akopyan}
\address{Arseniy~Akopyan, Institute of Science and Technology Austria (IST Austria), Am Campus~1, 3400 Klosterneuburg, Austria}
\email{akopjan@gmail.com}

\author{Roman~Fedorov}
\address{Roman~Fedorov, University of Pittsburgh, The Dietrich School of Arts and Sciences, 301 Thackeray Hall,
Pittsburgh, PA 15260, USA}
\email{rmfedorov@gmail.com }

\begin{document}


\begin{abstract}
We answer a question of David Hilbert: given two circles it is not possible in general to construct their centers using only a straightedge. On the other hand, we give infinitely many families of pairs of circles for which such construction is possible.
\end{abstract}

\maketitle

\section{Introduction}
The famous Mohr--Mascheroni Theorem~\cite{mascheroni1797geometria, mohr1928euclides} states that any geometric construction that can be performed by a compass and a straightedge (a.k.a a ruler without marks) can be performed by a compass alone. Of course, since a straight line cannot be drawn by a~compass, we encode a line by any two points lying on it. On the other hand, if we use only a straightedge, we cannot construct even perpendicular lines; this will be discussed in Section~\ref{sect:perp}.

In 1833 J.~Steiner~\cite{steiner1833geometrischen} proved the conjecture of J.~V.~Poncelet stating that if a circle with a center is drawn on the plane, then everything that can be constructed with a compass and a straightedge can be constructed with the straightedge alone.
(Since a circle cannot be drawn by a straightedge, we encode a circle by its center and a point lying on this circle.) It was noted by D.~Hilbert that the center of the drawn circle plays a crucial role because there is a three-parametric family of projective transformations preserving the circle. These transformations preserve straightedge constructions but may change the center of the circle, so we cannot ``catch'' the center only by a straightedge. We refer to \cite{introduction_steiner1950geometrical_englishedition} for historical remarks and a further discussion of the problem.

If two intersecting, tangent, or concentric circles are drawn on the plane, then one can find their centers using only a straightedge; we will recall these constructions in Section~\ref{sect:intcircles}. D.~Hilbert has asked whether it is possible in general to construct the centers of two circles with a straightedge only (see the beginning of~\cite{cauer1912ueber} and~\cite[i]{introduction_steiner1950geometrical_englishedition}). This question was answered in the paper of D.~Cauer~\cite{cauer1912ueber}. He observed that for two non-intersecting non-concentric circles there is a projective transformation preserving two circles as sets but changing their centers. Based on that, he claimed that there is no algorithm constructing the centers of circles, since if it existed, it would return the images instead of the actual centers.

Forty years later, C.~Gram~\cite{gram1956aremark} has found an error in the arguments of Cauer. He has also given a construction of the centers in the case when a point on the line connecting the centers is given (see Proposition~\ref{pr:Gram} below), despite the proof of Cauer goes through in this case as well. Gram's proof, in turn, contains a minor mistake, which we discuss and correct in Section~\ref{ss:gram theorem}.

The problem with Cauer's proof is that, in the first instance, to show that a construction algorithm does not exist, one needs a definition of construction algorithm. This question is subtler than it looks, see~\cite{ShenOnHilbertsError}. In particular, a reasonable definition of an algorithm must include the operation of testing whether two lines are parallel or not. This operation is not preserved under projective transformations. In Section~\ref{ss:special cases} we show that there exist infinitely many families of pairs of non-intersecting circles such that the construction of their centers is possible. Therefore Cauer's proof cannot be repaired. In fact, a similar problem occurs in Hilbert's theorem, see Section~\ref{sect:center}.

The main result of the article is the following theorem.
\begin{theorem}\label{thm: centers}
	There exist two circles whose centers cannot be constructed using only a~straightedge.
\end{theorem}

We will prove this theorem in Section~\ref{sect:main} for any `reasonable' definition of construction algorithm (to be discussed in Section~\ref{sect:non_constr}). In fact, we will provide an example of such circles. This theorem follows as a trivial corollary from the Steiner--Poncelet Theorem discussed above and the following theorem.

\begin{theorem}\label{thm: two tangents}
    There exist two circles, lying one outside the other, such that it is not possible to construct their common tangent lines using only a straightedge.
\end{theorem}

The idea of the proof is that, were it possible, it would be also possible to construct common tangents to any two conics as well because any two non-intersecting conics are projectively equivalent to two circles. However, the latter is not possible in general for field theory reasons. We give all the formal definitions in the following sections.

For the case of three circles the construction of the centers is possible if the three circles do not belong to the same pencil. The first proof of this statement, also given by Cauer in~\cite{cauer1912ueber}, was incorrect. The mistake was found by Schur and Mierendorff, their alternative construction appeared in \cite{cauer1913ueber}. In Sections~\ref{sect:threecircles}--\ref{ss:gram theorem} we sketch an algebraic proof of this statement.

\begin{remark}
A. Shen has pointed out that Cauer was aware of a strategy similar to the one we follow to prove Theorem~\ref{thm: centers} (see~\cite{cauer1913ueber}). However, Cauer did not give any details.
\end{remark}

\section{Definition of non-constructible points}\label{sect:non_constr}
As we have already mentioned, to show that a certain point is non-constructible, we need a definition of the geometric construction. It turns out that this question is quite subtle. There are many sources where definitions of geometric constructions are discussed, see e.g.~\cite{manin1963solvability,engeler1968remarks}, and especially~\cite{ShenOnHilbertsError}).

The list of operations in~\cite{engeler1968remarks} includes the operation of checking whether lines are parallel or not (B2) and an operation ``if''\footnote{Constructions in the Mohr--Mascheroni and the Steiner--Poncelet theorems also use this operation.}. On the other hand, this list does not include an operation of ``taking a random point''. Clearly, we need to assume that we can pick a random point, otherwise, starting from two circles we cannot get anywhere. Random points can be chosen in certain open regions or on arcs of curves. Random point on arcs can be defined as the intersection points of the curve with segments connecting two random points in areas near the arc, so we can work with random point in open regions only. The randomness means that any point in the open region will do. Thus, we may assume that there is an adversary playing against us by choosing a dense subset of the plane and requesting that random points are taken from this subset. Indeed, a dense subset intersects any open region.

To avoid a further discussion of algorithms, we define a concept of ``non-constructible'' points. These points cannot be constructed by any reasonable construction algorithm. For completeness, in Section~\ref{sect:further} we give a definition of a ``general algorithm'', and show that point is non-constructible if and only if the general algorithm does not construct it.

\begin{definition}\label{def:main}
Given a configuration of points, lines, and curves $\{a_i,\ell_i,\gamma_i\}$, we say that a point $a\in\R^2$ is \emph{non-constructible} from the configuration by a straightedge, if there is a set $\Sigma\subset\R^2$ satisfying the following properties:
\begin{enumerate}
\item\label{def:main:i} $\Sigma$ is dense in $\R^2$.
\item \label{def:main:ii}  For all $i$ we have $a_i\in\Sigma$.
\item\label{def:main:iii} Given points $b_1,\ldots,b_4\in\Sigma$, the intersection of the lines $b_1b_2$ and $b_3b_4$ is contained in~$\Sigma$, provided these lines are distinct and not parallel.
\item\label{def:main:iv} Given points $b_1,b_2\in\Sigma$, the isolated intersection points of the line $b_1b_2$ with the lines $\ell_i$ and the curves $\gamma_j$ are in $\Sigma$.
\item\label{def:main:v} The point $a$ is not in $\Sigma$.
\end{enumerate}
\end{definition}

Let us clarify, why there is no algorithm constructing the point $a$ from the configuration $\{a_i,\ell_i,\gamma_i\}$, provided that $a$ is non-constructible. Indeed, suppose such an algorithm exists. As explained above, we may assume that random points are always taken from $\Sigma$, thanks to property~\eqref{def:main:i} of~$\Sigma$. An easy inductive argument using properties~\eqref{def:main:iii} and~\eqref{def:main:iv} of $\Sigma$ shows that all the points we construct by our algorithm are contained in $\Sigma$. But this gives a~contradiction with~\eqref{def:main:v}.

\begin{remark}\label{rem:non-constractable defintiion}
(i) Conversely, if $a$ is constructible from $\{a_i,\ell_i,\gamma_i\}$ (that is, not non-constructible), then there is an algorithm constructing $a$ from $\{a_i,\ell_i,\gamma_i\}$; see Section~\ref{sect:further}.

(ii) The set $\{a_i\}$ is not required to contain all intersection points of the lines $\{\ell_i\}$ and the curves $\{\gamma_i\}$; cf.\ Remark~\ref{rm:mainthm}(i).

(iii) In a similar way we can define ``non-constructible'' lines or configurations of lines by substituting part~\eqref{def:main:v} with the requirement that among the set of lines connecting points in~$\Sigma$ we cannot find those forming the needed configuration.

(iv) If we have four point in general position, we can cover the plane with a set of lines whose pairwise intersection points are dense in the plane. Then we can pick our random points from this set (cf.~Section~\ref{ss:midpoint} below). Thus we do not actually need random points in this case. Similarly, if we have a straightedge and a compass, we can get a dense set, starting from any two points.
\end{remark}

\section{Known results}\label{sec:known results}
In this section we list some well-known non-constructibility results and show how to give simple proofs using Definition~\ref{def:main}. In the end of the section we explain how to construct the centers of two intersecting, tangent, or concentric circles. A reader who is only interested in the proofs of our main theorems may skip this section.

\subsection{Midpoint of a segment}\label{ss:midpoint}
The problem of constructing a line parallel to a given line is equivalent to constructing the midpoint of a segment. Indeed, having a midpoint of a~segment, we can construct the line parallel to this segment and passing through a given point; see Figure~\ref{fig:construction of parallel line}.

Conversely, having two parallel lines, we can find a midpoint of any segment parallel to these lines. There is no algorithm solving any of these problems using only a~straightedge. The well-known proof is the following: assume that there is an algorithm constructing the midpoint. Consider a projective transformation $\Pi$ taking the segment to itself but moving the midpoint. Then this transformation preserves the algorithm giving a contradiction. The~gap in this proof was found by V.~J.~Baston and F.~A.~Bostock \cite{baston1990ontheimpossibility}: as explained above, the algorithm can contain testing whether two lines are parallel or not, and the property of being parallel is not invariant under projective transformations. Moreover they noted that if instead of $\R^2$ we consider the rational plane $\Q^2$, then an algorithm exists! Indeed, choose four points in general position and a projective coordinate system such that these points are the vertices of the unit square. Using a construction similar to that of Figure~\ref{fig:construction of parallel line}, it is easy to construct the point with given rational coordinates. Now, using the fact that $\Q^2$ is a countable set, it is not difficult to design an algorithm returning successively all the points of $\Q^2$.

\begin{center}
	\includegraphics{fig-ruler-3.mps}\\
	
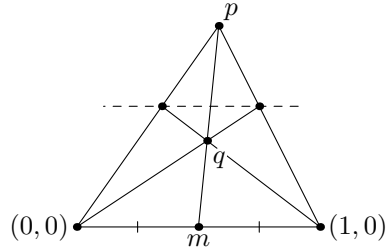
	\f \label{fig:construction of parallel line} Construction of a line parallel to the segment $[(0,0),(1,0)]$ using its midpoint.
\end{center}

Baston and Bostock have repaired the proof, but let us show how this can be done using our Definition~\ref{def:main}.

\begin{proposition}
The midpoint of a segment cannot be constructed with only a straightedge.
\end{proposition}
\begin{proof}
Without loss of generality, we may assume that the segment is $[(0,0),(1,0)]$ (otherwise, make an affine coordinate change). Let us take a projective transformation $\Pi$ such that $\Pi$ takes the segment to itself and the midpoint $(0,1/2)$ to a point $(0,x)$, where $x$ is irrational. Let $\Sigma$ be the set of points $a\in\R^2$ such that $\Pi(a)$ has rational coordinates\footnote{$\Pi(a)$ may be a point on the line at infinity; cf.~the proof of Theorem~\ref{thm: two tangents} below.}. Clearly, $\Sigma$ satisfies the conditions of Definition~\ref{def:main}. Since the midpoint is not in $\Sigma$, it cannot be constructed from the configuration $\{(0,0),(1,0)\}$.
\end{proof}

This proof also shows that for any field $\mathbb F$ strictly containing $\Q$, one cannot construct the midpoint of a segment.

\bigskip

By the same reason as in the midpoint problem, for the next two problems the standard proofs based only on a projective transformation are not completely correct. Indeed, for the rational plane $\Q^2$ the next two problems have positive solutions as well.

\subsection{Center of a circle}\label{sect:center}
Now we give a (first correct?) proof of a well-known theorem of Hilbert.
\begin{proposition}[Hilbert]\label{prop:hilbert}
Given a circle it is not possible to construct its center.
\end{proposition}
The usual proof of this theorem is to notice that there is a projective transformation preserving the circle but moving the center. This proof contains the same mistake as Cauer's proof and the standard proof of non-constructibility of midpoint of a segment from the previous section: projective transformations neither preserve parallelity, nor preserve order of points on lines.
\begin{proof}[Proof of Proposition~\ref{prop:hilbert}]
Consider the coordinate system such that our circle is the unit circle. Let $\G$ be the subfield of complex numbers consisting of all numbers that can be obtained from rational numbers via iterated application of field operations and extraction of square roots. Let $\Pi$ be a projective transformation preserving a circle but moving its center to a point whose coordinates are not in $\G$. Then we take $\Sigma$ to be the set of points $a$ such that the coordinates of $\Pi(a)$ are in $\G$. We see that, according to Definition~\ref{def:main}, the center cannot be constructed.
\end{proof}
\begin{remark}
Note that in Cauer's case, there is only one projective transformation preserving the two circles but moving the centers. If we were able to reduce the possibilities for the centers to finite number of points, we could easily decide, which point is actually the center.
In Hilbert's case there are uncountably many transformations preserving the circle but moving the center. It seems that if we have a configuration and an uncountable family of projective transformation preserving the configuration but moving a certain point, then this point cannot be constructed (use arguments similar to~\cite[Sect.~4]{ShenOnHilbertsError}). On the other hand, we do not know if there is a naive way to repair Hilbert's proof without using the notion of non-constructible point (our Definition~\ref{def:main}).
\end{remark}

\subsection{Perpendicular lines}\label{sect:perp}
\begin{proposition}
It is not possible to construct a pair of perpendicular lines using only a straightedge.
\end{proposition}
\begin{proof}
Take $\Sigma=\Pi(\Q^2)$, where $\Pi$ is the affine transformation taking $(x,y)$ to $(x,x+\sqrt3y)$. We leave it to the reader to prove that no two lines connecting points of $\Sigma$ are perpendicular (cf.~Remark~\ref{rem:non-constractable defintiion}(iii)). (Hint: the slope of such a line belongs to $1+\Q\sqrt3$, the product of two numbers in this set cannot be equal to $-1$).
\end{proof}

\subsection{Intersection of a line and a circle} Can one construct the intersection of a line and a circle, if the line is given, but only $100$ points are given on the circle? If the line passes through one of the points, then the answer is ``yes'' by Pascal's Theorem. In general the answer is ``no'' because the line can have a rational equation and points on the circle can have rational coordinates, but the intersection of the line and the circle can have irrational coordinates. The example of this situation is the unit circle with $100$ rational points on it and the line $y=x$, whose intersection with the circle consists of two points $\pm\left(\frac{\sqrt2}2, \frac{\sqrt2}2\right)$. In this case, one takes $\Sigma=\Q^2$ in Definition~\ref{def:main}.

\subsection{Intersecting and concentric circles}\label{sect:intcircles}
It was noted in Cauer's paper that if two circles intersect (including the case of touching circles) or are concentric, then it is possible to construct their centers. The construction of the centers of two intersecting circles is based on a simple observation shown on Figure~\ref{fig:parallel lines}: we can construct two pairs of parallel lines and therefore find the midpoints of the chords they cut from the circles. Then it is not difficult to construct two pairs of diameters of circles and find their centers. The construction for tangent circles is similar.

For a pair of concentric circles the construction is shown on Figure~\ref{fig:concentric circles}. It is based on the fact that, given a circle, we can construct a tangent line to this circle through any point on or outside the circle. We refer to the book of A.~S.~Smogorzhevskii~\cite[\S19]{smogorzhevskiui1961ruler} for detailed explanations.

\parbox[b]{8cm}{
\begin{center}
	\includegraphics{fig-ruler-4.mps}\\
	
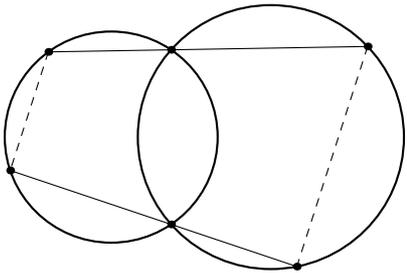
	\f \label{fig:parallel lines} Construction of two parallel lines from two intersecting circles.
\end{center}}
\parbox[b]{8cm}{\begin{center}
	\includegraphics{fig-ruler-5.mps}\\
	
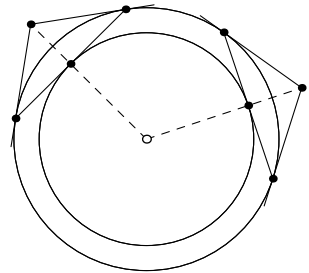
	\f \label{fig:concentric circles} Sketch of construction of the common center of two concentric circles.
\end{center}}

\section{Main argument}\label{sect:main}

\subsection{Proof of Theorem \ref{thm: two tangents}}
Consider the circle $\alpha$ given by the equation $(x-1)^2+(y+1)^2=1$ and the parabola $\pi$ given by $y=x^2$. We view our plane $\R^2$ as the subset of the projective plane $\RP$ given in standard projective coordinates $(X:Y:Z)$ by the equation $Z\ne0$. Let $\Pi:\RP\to\RP$ be a projective transformation, whose complexification $\Pi_{\C}:\CP\to\CP$ takes the two cyclic points $(1:\pm\sqrt{-1}:0)$ to complex conjugate intersection points of the circle and the parabola. Recall that an infinite real conic is a circle if and only if it passes through the cyclic points. Thus $\Pi^{-1}$ takes the circle $\alpha$ and the parabola $\pi$ to two circles; denote them by $\omega_1$ and $\omega_2$ respectively. We claim that the common tangents to these two circles cannot be constructed by a straightedge.

Recall that $\G$ denotes the subfield of the field of complex numbers consisting of all numbers that can be obtained from rational numbers via iterated application of field operations and extraction of square roots.

One common tangent line to $\alpha$ and $\pi$ is the $x$-axis $y=0$. Consider the points $a_1$, $a_2$, and~$a_3$ where the remaining three common tangent lines of $\alpha$ and $\pi$ touch $\alpha$.
\begin{lemma}
    The $x$-coordinates of $a_1$, $a_2$, and $a_3$ do not belong to $\G$.
\end{lemma}

Let us finish the proof of the theorem, assuming the lemma. Let our configuration consist of the circles $\omega_1$ and $\omega_2$ (and there are no points and lines in the configuration). Define the set $\Sigma$ as the set of points $a\in\R^2$ such that the coordinates of $\Pi(a)$ are in $\G$. (The coordinates of a point in~$\RP$ are defined up to scaling, so when we say that the coordinates are in $\G$, we mean that they are in $\G$ after appropriate scaling.)

It is easy to see that $\Sigma$ satisfies conditions~\eqref{def:main:i}--\eqref{def:main:iii} of Definition~\ref{def:main}. To check condition~\eqref{def:main:iv} we note that for points $b_1,b_2\in\Sigma$, the line $\Pi(b_1b_2)$ has an equation with coefficients in $\G$, so its intersection points with $\alpha$ and $\pi$ have coordinates also in $\G$.

Consider the points $b_1$, $b_2$, and $b_3$, where the common tangents to $\omega_1$ and~$\omega_2$ are tangent to~$\omega_1$. Their images under $\Pi$ are the points $a_1$, $a_2$, and $a_3$. Now it follows from the above lemma that $b_1$, $b_2$, and $b_3$ are not in $\Sigma$. According to Definition~\ref{def:main}, $b_1$, $b_2$, and $b_3$ are non-constructible from the configuration $\{\omega_1,\omega_2\}$. Thus, the common tangents cannot be constructed as well.

\begin{proof}[Proof of the lemma]
Consider a line $\ell_x$ with equation $y=t^2+2t(x-t)$, which is tangent to the parabola~$\pi$ at the point $(t,t^2)$.
The $x$-coordinates of intersections of this line and the circle $\alpha$ are solutions of the equation
\begin{equation*}
	1=(x-1)^2+(t^2+2t(x-t)+1)^2=(4t^2+1)x^2-(4t(t^2-1)+2)x+(t^2-1)^2+1.
\end{equation*}
The line $\ell_x$ tangents $\alpha$ if and only if the determinant of this quadratic equation is equal to $0$:
\begin{equation*}
D(t)=(4t(t^2-1)+2)^2-4(4t^2+1)(t^2-1)^2=-4t(t^3-4t^2-2t+4)=0.
\end{equation*}
Thus the $x$-coordinates of the points $a_1$, $a_2$, and $a_3$ are the roots of the polynomial $P(t)=t^3-4t^2-2t+4$. We claim that $P(t)$ has no rational roots. Indeed, this polynomial is monic, therefore all its rational roots should be integer and divide 4. It remains to check that $\pm1$, $\pm2$, and $\pm4$ are not roots of the polynomial (the roots are equal approximately $0.85$, $4.25$, and $-1.10$).

It remains to apply the following well-known claim, whose proof we give for completeness, as we cannot find a reference.
\begin{claim*}
    If the roots of a cubic polynomial having rational coefficients are irrational, then they do not belong to $\G$.
\end{claim*}
\begin{proof}
	If a cubic polynomial has no rational roots, then it is irreducible. Thus for any root~$\beta$ we have $[\Q(\beta):\Q]=3$, where $[\Q(\beta):\Q]$ is the degree of the extension (see e.g.~\cite[\S13]{DummitFooteAlgebra}). However, if $\beta\in\G$, then~$\beta$ would belong to an extension $F\supset\Q$ such that $[F:\Q]=2^n$ for some $n\in\Z_+$. We get a contradiction with $[F:\Q]=[F:\Q(\beta)][\Q(\beta):\Q]$. The proof of the claim is complete.
\end{proof}

The claim completes the proof of the lemma.
\end{proof}

\begin{remark}\label{rm:mainthm}
(i) Assume that we have a configuration of two intersecting circles, but we do not have the intersection points (cf.~Remark~\ref{rem:non-constractable defintiion}(ii)). Then, similarly to the above, it is easy to give an example where the centers (and thus the intersection points) cannot be constructed. In other words, if we erase small neighbourhoods of intersection points, then the centers cannot be in general constructed. On the other hand, given a small arc of a circle one can construct the intersection of the circle with any line using a straightedge only, cf.~\cite{mordouhay-boltovskoy1910Ogeometricheskih}.

\end{remark}

\subsection{Special cases}\label{ss:special cases}
We have proved that for \emph{some\/} circles their centers cannot be constructed. However, for \emph{many\/} pairs of circles they can be constructed. We have already seen in Section~\ref{sect:intcircles} that we can construct the centers of intersecting, tangent, or concentric circles. We are going to give more examples now. By Gram's Theorem (see Proposition~\ref{pr:Gram}) it is enough to find one point on the line connecting the centers. For some cases this point can be constructed with the help of Poncelet's Porism. Let us recall the statement we need: \emph{if a polygon is inscribed in a circle~$\omega_1$ and circumscribes a circle $\omega_2$, then for any point $a$ on $\omega_1$ there is a polygon with the same number of sides with vertex at~$a$, inscribed in $\omega_1$ and circumscribed around $\omega_2$. Moreover, if the number of sides is even, then the main diagonals intersect at a fixed point, lying on the line connecting the centers of $\omega_1$ and $\omega_2$; see~\cite[\S IV.8]{BergerLadder}}.

We see that if there is a polygon inscribed in one circle and circumscribing the other circle, and this polygon has even number of sides, then the centers of the circles can be constructed. An example of quadrilaterals is shown on Figure~\ref{fig:poncelet circles}(left). Note that this also works for self-intersecting polygons (a.k.a.~polylines), and the word ``circumscribed'' should be understood in the generalized sense, see Figure~\ref{fig:poncelet circles}(right).

\begin{remark}
There are many more families of pairs of circles for which construction of their centers is possible. For example, it can be shown that if a polyline (not necessarily closed) rotates between circles $\omega_1$ and~$\omega_2$, then the intersections of its sides and diagonals move along conics. Moreover, using only a straightedge it is possible to find intersections of these conics with any line on the plane. Now we can consider polylines, circumscribed and inscribed between any pair of these conics, and generate more conics, and so on. One can show that all these conics have a common axis of symmetry---the line connecting the centers of the circles. If between some pair of these conics one can ``circum-inscribe'' a closed polyline with even number of edges, then its main diagonals intersect on this axis, that is, on the line connecting the centers of $\omega_1$ and $\omega_2$. Therefore we can construct these centers by Gram's Theorem. Note that we could get even a larger family of conics by considering dual images of these conics with respect to each other. We could also consider polylines touching not a fixed conic, but different conics from a pencil.

We do not know of a good way to describe all pairs of circles coming from this construction. We also do not know, if there are pairs of circles for which the centers can be constructed, not coming from the construction of this remark.
\end{remark}

\begin{center}
	\includegraphics{fig-ruler-6.mps}\hskip 1cm
	\includegraphics{fig-ruler-7.mps}\\
	
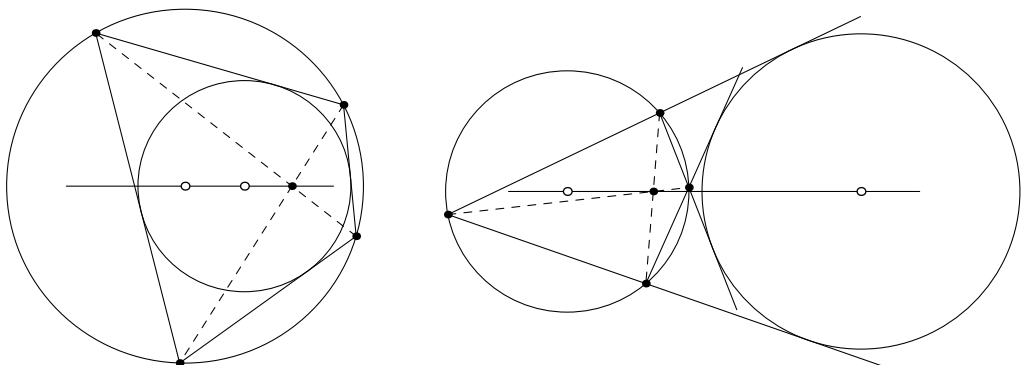
	\f \label{fig:poncelet circles} Pairs of circles for which exists ``circum-inscribed'' quadrilaterals. Their diagonals intersect on the line connecting the centers.
\end{center}

\subsection{Pencil of circles}\label{ss: pencil of circles}
If we are given many circles from a single pencil, then, in general, we cannot construct the common tangents. To see it, consider the circles $\omega_1$ and $\omega_2$ from the proof of Theorem~\ref{thm: two tangents}. Now consider any number of circles from the pencil of $\omega_1$ and $\omega_2$ such that after the projective transform $\Pi$ these circles map to a conic with equation obtained as a linear combination of equations of $\alpha$ and $\pi$ with coefficients in $\G$. Repeating the argument in the proof of Theorem~\ref{thm: two tangents}, we see that we cannot construct the common tangents to $\omega_1$ and~$\omega_2$.

\subsection{Three circles not from the same pencil}\label{sect:threecircles}
On the other hand, if we are given three circles not from one pencil, then their centers can be constructed, as shown in~\cite{cauer1913ueber}. Let us sketch an algebraic argument. Suppose we are given circles $\omega$, $\alpha_1$ and $\alpha_2$ not from the same pencil. Choose any point $o$ inside $\omega$ and a line $\ell$ through $o$. Consider a system of projective coordinates such that~$o$ is the origin and the center of $\omega$, and $\ell$ is the $x$-axis. Note that in this coordinate system $\alpha_1$ and $\alpha_2$ are two conics. Now, by the Steiner--Poncelet Theorem we may assume that in addition we have a compass, and therefore we can construct the $y$-axis and find the coefficients of the equations of $\alpha_1$ and $\alpha_2$. Let $P_i=0$ ($i=1,2$) be the degree four equation for the $x$-coordinates of the complex intersection points of $\omega$ and $\alpha_i$. Since all three circles pass through two cyclic points, $P_1$ and $P_2$ have a quadratic common factor ($P_1$ and $P_2$ are not proportional because the circles are not from the same pencil). This factor can be found by Euclidean algorithm. By factorization, solving $P_1$ and $P_2$ reduces to solving quadratic equations. Since all these operations can be performed with a compass and a straightedge, we can construct the complex coordinates (that is, their real and imaginary parts) of the two non-cyclic intersection points of $\alpha_1$ and $\omega$. Now we can repeat the ``real construction'' from Section~\ref{sect:intcircles} for the non-intersecting circles $\alpha_1$ and $\omega$, working formally with complex coordinates of their intersection points.

\subsection{Algebraic proof of Gram's Theorem}\label{ss:gram theorem}
Now we give an algebraic proof of Gram's theorem, explaining why having a point on the central line is crucial to the problem. Then we discuss and correct a gap in the original Gram's proof.

\begin{proposition}[Gram's Theorem]\label{pr:Gram}
Given two circles and a point on the line connecting centers of the circles, the centers can be constructed with the straightedge only.
\end{proposition}
\begin{proof}
Let the circles $\omega_1$ and $\omega_2$ and a point $a$ on the line connecting the centers be given. Consider first the case, when $a$ is inside one of the circles, say, inside $\omega_1$. There is a projective coordinate system such that $\omega_1$ is the unit circle, $a$ is its center and the $x$-axis is the symmetry axis of $\omega_2$ (which is a conic in this coordinate system).

To construct the axis of this coordinate system we first, as in the previous section, construct the axis of a coordinate system such that the origin $a$ is the center of the unit circle~$\omega_1$ (this differs from the required coordinate system by a rotation). As in Section~\ref{sect:threecircles}, we may assume that we have a compass. In this coordinate system we can find the coefficients of the equation of $\omega_2$. Then it is not difficult to find the symmetry axis of the conic. One of the symmetry axis is the required $x$-axis of the sought-for coordinate system.

In this coordinate system, the equation of $\omega_2$ has zero-coefficients at the monomials $y$ and~$xy$. Now, consider the systems of equations for the (complex) intersection points of the conic and the circle. To solve this system, we substitute $1-x^2$ for $y^2$. We get a quadratic equation in $x$; solving the equation we find the (complex) $x$-coordinates of the intersection points of $\omega_1$ and $\omega_2$. Next, we can construct all four complex points of intersection, two of which lie on the line at infinity (in the original coordinate system). The poles of this line with respect to the circles will be the centers of the circles; they can be now easily constructed by a straightedge.

If $a$ is outside the circles (Figure~\ref{fig:construction of center-line}), the tangents from $a$ to one of the circles intersect (or touch) the second circle. Connecting the corresponded points of intersection of tangents with circles we get two parallel lines
perpendicular to the line of the centers. Hence, we can construct the midpoints of the chords they cut from the circles (the construction is reverse to the one described in Section~\ref{ss:midpoint}). These midpoints are inside the circles and on the line connecting the centers, so it remains to argue as in the previous paragraphs.

\end{proof}
\begin{center}
	\includegraphics{fig-ruler-8.mps}\\
	
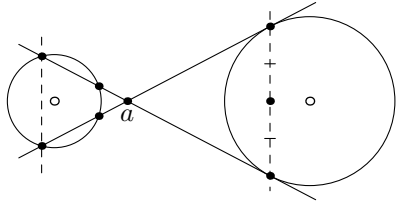
	\f \label{fig:construction of center-line} Construction of a point on a line through the centers of two circles, lying inside one of the circles.
\end{center}

\begin{remark}
The first step of Gram's proof of his theorem is to construct the line $\ell$ connecting the centers of $\omega_1$ and $\omega_2$: taking polars of $a$ with respect to $\omega_1$ and $\omega_2$ we get two parallel lines, perpendicular to $\ell$. Using them, one can construct two chords $b_1b_2$ and $c_1c_2$ of $\omega_1$ perpendicular to $\ell$. Connecting the point $a$ with intersection of $b_1c_1$ and $b_2c_2$ we get the line~$\ell$.

It was pointed to us by Alexey Zaslavsky, that if $a$ is one of the limit points of the pencils generated by $\omega_1$ and $\omega_2$ this construction does not work: the polar lines of $a$ with respect to $\omega_1$ and $\omega_2$ coincide. Note, that this is exactly the case we have in Section~\ref{ss:special cases}! This is a minor problem in the algebraic proof. The Gram's proof can also be repaired as follows. In his paper Gram shows, that using a straightedge we can ``emulate'' any circle from the pencil generated by $\omega_1$ and $\omega_2$. This means that for any point $p$ one can find the intersection points of any line with the circle from the pencil passing through a given point~$p$. Choose circles $\omega'$ and $\omega''$ from the pencil so that $a$ is outside of them. Now we can repeat arguments from the algebraic proof.
\end{remark}

\section{Further discussion of construction algorithms}\label{sect:further}
Consider the following construction algorithm. At the first step, starting from a configuration $\{a_i,\ell_i,\gamma_i\}$, the algorithm requests four random points from an adversary. The only requirement is that these random points should be in general position, that is, no three points are on the same line and no two segments formed by these points are parallel. (These points are needed for generating an everywhere dense set, from which random points can be selected; see the discussion before Definition~\ref{def:main}). At every next step one adds the lines connecting all pairs of points in the configuration, all the intersection points of lines in the configuration with other lines in the configuration, and the isolated intersection points of the lines with the curves. A point is constructed, if it is in the configuration at some step. We call it a \emph{general algorithm}. We note that the general algorithm does not involve any choices, except by the adversary at the first step.

\begin{proposition}\label{pr:defAlg}
Given a configuration $\{a_i,\ell_i,\gamma_i\}$ of points, lines, and curves, exactly one of the following two statements is true.\\
(i) The point $a$ is non-constructible from the configuration in the sense of Definition~\ref{def:main}.\\
(ii) The general algorithm constructs the point~$a$.
\end{proposition}
\begin{proof}
The paragraph after Definition~\ref{def:main} shows that if $a$ is non-constructible, then it is not constructed by the general algorithm.

Conversely, assume that $a$ is not non-constructible (that is, there is no $\Sigma$ as in Definition~\ref{def:main}). Let $\Sigma$ be the set of all points obtained by the general algorithm. Clearly, it satisfies conditions~\eqref{def:main:ii}--\eqref{def:main:iv} of Definition~\ref{def:main}. Since it contains four given by the adversary points in general position, it is everywhere dense (condition~\eqref{def:main:i}). Thus, by the assumption of the Theorem, $\Sigma$ does not satisfy condition~\eqref{def:main:v}, that is, $a\in\Sigma$.	
\end{proof}

\begin{remark}
We assume that the point $a$ can be recognized, once it is constructed. Note that sometimes the existence of such a test depends on the definition of algorithm. For example we can ``see'' if two lines are parallel, which can be used to recognize the center of a circle. But can we ``see'' that two lines are perpendicular? We prefer to avoid further discussion as it would take us too far.
\end{remark}

\textbf{Acknowledgments.}
The authors have learnt about the problem from Sergey Markelov and Alexander Shen. They are grateful to Sergey for his persistent interest in the problem and useful advice and to Alexander for stimulating discussions about subtleties of straightedge constructions. They thank Alexey Zaslavsky for pointing out a flaw in Gram's proof of his theorem and G\"unter Ziegler for numerous remarks.

The idea of the proof of the main theorem was found during a Summer school ``Modern Mathematics'' held in Dubna, Russia.


\begin{thebibliography}{10}

\bibitem{introduction_steiner1950geometrical_englishedition}
R.~C. Archibald and M.~E. Stark.
\newblock Editor's introduction.
\newblock In {\em Jacob Steiner, Geometrical {C}onstructions with a {R}uler,
  {G}iven a {F}ixed {C}ircle with {I}ts {C}enter}, Scripta Mathematica Studies,
  no. 4, pages 1--9. Scripta Mathematica, New York, N. Y., 1950.

\bibitem{baston1990ontheimpossibility}
V.~J. Baston and F.~A. Bostock.
\newblock On the impossibility of ruler-only constructions.
\newblock {\em Proc. Amer. Math. Soc.}, 110(4):1017--1025, 1990.

\bibitem{BergerLadder}
M.~Berger.
\newblock {\em Geometry revealed. A Jacob's ladder to modern higher geometry}.
\newblock Springer, Heidelberg, 2010.
\newblock Translated from the French.

\bibitem{cauer1912ueber}
D.~Cauer.
\newblock {\"Uber die {K}onstruktion des {M}ittelpunktes eines {K}reises mit
  dem {L}ineal allein}.
\newblock {\em Math. Ann.}, 73(1):90--94, 1912.

\bibitem{cauer1913ueber}
D.~Cauer.
\newblock {\"Uber die {K}onstruktion des {M}ittelpunktes eines {K}reises mit
  dem {L}ineal allein} ({B}erichtigung).
\newblock {\em Math. Ann.}, 74(3):462--464, 1913.

\bibitem{DummitFooteAlgebra}
D.~S. Dummit and R.~M. Foote.
\newblock {\em Abstract algebra}.
\newblock John Wiley \& Sons, Inc., Hoboken, NJ, third edition, 2004.

\bibitem{engeler1968remarks}
E.~Engeler.
\newblock Remarks on the theory of geometrical constructions.
\newblock In {\em The syntax and semantics of infinitary languages}, Edited by
  Jon Barwise. Lecture Notes in Mathematics, No. 72, pages 64--76. Springer,
  1968.

\bibitem{gram1956aremark}
C.~Gram.
\newblock A remark on the construction of the centre of a circle by means of
  the ruler.
\newblock {\em Math. Scand.}, 4:157--160, 1956.

\bibitem{manin1963solvability}
{\relax Ju}.~I. Manin.
\newblock On solvability of problems on construction by compass and ruler.
\newblock In P.~S. Alexandrov, A.~I. Markushevich, and A.~{\relax Ya}. Hinchin,
  editors, {\em Encyclopedia of elementary mathematics}, volume 4, (Geometry).
  Moscow, Fizmalit, 1963.

\bibitem{mascheroni1797geometria}
L.~Mascheroni.
\newblock {\em La Geometria del {C}ompasso}.
\newblock {P}resso gli eredi di Pietro Galeazzi, 1797.

\bibitem{mohr1928euclides}
G.~Mohr.
\newblock {\em Euclides Danicus}.
\newblock Amsterdam: Jacob van Velsen, 1672.

\bibitem{mordouhay-boltovskoy1910Ogeometricheskih}
D.~D. Mordouhay-Boltovskoy.
\newblock On geometric constructions with a straightedge, under the condition
  that an arc of a circle with the center is given.
\newblock {\em Vestnik Opitnoy Fiziki i Elementarnoy Matematiki},
  (522):137--146, 1910.
\newblock Russian.

\bibitem{ShenOnHilbertsError}
A.~Shen.
\newblock Hilbert's error?
\newblock {\em In preparation}.

\bibitem{smogorzhevskiui1961ruler}
A.~S. Smogorzhevskii.
\newblock {\em The ruler in geometrical constructions}, volume~5 of {\em
  Popular lectures in mathematics}.
\newblock Blaisdell Pub. Co., 1961.

\bibitem{steiner1833geometrischen}
J.~Steiner.
\newblock {\em Die geometrischen {K}onstructionen, ausgef{\"u}hrt mittelst der
  geraden {L}inie und eines festen {K}reises, als {L}ehrgegenstand auf
  h{\"o}heren {U}nterrichts-{A}nstalten und zur praktischen {B}enutzung}.
\newblock Berlin: Ferdinand D{\"u}mmler, 1833.

\end{thebibliography}
\end{document}